\documentclass[12pt]{article}
\usepackage{amsfonts}
\usepackage{amsmath}
\usepackage{amssymb}
\usepackage[mathscr]{eucal}
\usepackage{fullpage}
\usepackage[usenames]{color}
\usepackage{hyperref}    

\def\eod{\vrule height 6pt width 5pt depth 0pt}
\newenvironment{proof}{\noindent {\bf Proof:} \hspace{.2em}}
{\hspace*{\fill}{\eod}}

\newcommand{\floor}[1]{\left\lfloor #1 \right\rfloor}

\newcommand{\maj}{\mathsf{maj}}
\newcommand{\exc}{\mathsf{exc}}
\newcommand{\antiexc}{\mathsf{nexc}}
\newcommand{\nexc}{\mathsf{nexc}}
\newcommand{\sgn}{\mathsf{sign}}
\newcommand{\wexc}{\mathsf{wkexc}}
\newcommand{\ExcSa}{\mathsf{ExcSet}}
\newcommand{\des}{\mathsf{des}}
\newcommand{\asc}{\mathsf{asc}}

\newcommand{\DescSet}{\mathsf{Des\_Set}}
\newcommand{\ascSet}{\mathsf{Asc\_Set}}
\newcommand{\inv}{\mathsf{inv}}
\newcommand{\id}{\mathsf{id}}
\newcommand{\cyc}{\mathsf{cyc}}
\newcommand{\nest}{\mathsf{nest}}
\newcommand{\stat}{\mathsf{stat}}

\newcommand{\np}{ \mathsf{pos\_n}}
\newcommand\twoonethree{\operatorname{2-13}}
\newcommand\threeonetwo{\operatorname{31-2}}

\newcommand{\SSS}{\mathfrak{S}}
\newcommand{\BB}{\mathfrak{B}}

\newcommand{\DD}{\mathfrak{D}}

\newcommand{\Der}{\mathfrak{SD}}
\newcommand{\SD}{\mathfrak{SD}}
\newcommand{\nhalf}{\lfloor n/2 \rfloor}

\newcommand{\nmrhalf}{\lfloor (n-r)/2 \rfloor}
\newcommand{\AAA}{\mathcal{A}}

\newcommand{\PP}{ \mathrm{PalindPoly}}
\newcommand{\cnos}{\mathrm{ \operatorname{c-o-s}}}

\newcommand{\SB}{\mathrm{SgnB}}
\newcommand{\GE}{ \mathsf{AExc}}
\newcommand{\ADE}{ \mathsf{ADerExc}}

\newcommand{\cyctyp}{ \mathsf{cyctype}}
\newcommand{\EE}{ \mathsf{BExc}}
\newcommand{\EEsgnE}{ \mathsf{SgnBExc}}

\newcommand{\FFsgnE}{ \mathsf{SgnDExc}}
\newcommand{\FF}{ \mathsf{DExc}}

\newcommand{\FFBminusD}{ \mathsf{(B\mbox{-}D)Exc}}

\newcommand{\Negs}{\mathsf{Negs}}
\newcommand{\FFT}{\mathsf{FFT}}
\newcommand{\lngt}{\mathsf{len}}

\newcommand{\pos}{ \mathsf{pos}}

\newtheorem{theorem}{Theorem}

\newtheorem{corollary}[theorem]{Corollary}

\newtheorem{remark}[theorem]{Remark}

\newtheorem{lemma}[theorem]{Lemma}

\newcommand{\ZZ}{ \mathbb{Z}}
\newcommand{\QQ}{ \mathbb{Q}}

\newcommand{\comment}[1]{}

\newcommand{\ob}[1]{\overline{#1}}

\begin{document}
\title{Gamma positivity of the Excedance based Eulerian
polynomial in positive elements of Classical Weyl Groups}

\author{Hiranya Kishore Dey\\
Department of Mathematics\\
Indian Institute of Technology, Bombay\\
Mumbai 400 076, India.\\
email: hkdey@math.iitb.ac.in
\and
Sivaramakrishnan Sivasubramanian\\
Department of Mathematics\\
Indian Institute of Technology, Bombay\\
Mumbai 400 076, India.\\
email: krishnan@math.iitb.ac.in
}

\maketitle

\section{Introduction}
\label{chap:intro}
Let $f(t) \in \QQ[t]$ be a degree $n$ univariate polynomial with 
$f(t) = \sum_{i=0}^n a_i t^i$ where $a_i \in \QQ$ with 
$a_n \not= 0$.  Let $r$ be the least non-negative integer such that 
$a_r \not= 0$.  Define $\lngt(f) = n-r$.  $f(t)$ is 
said to be palindromic if $a_{r+i} = a_{n-i}$ for $0 \leq i \leq (n-r)/2$.  
Define the {\it center of 
symmetry} of $f(t)$ to be $\cnos(f(t)) =  (n+r)/2$.  
Note that for a palindromic 
polynomial $f(t)$, its center of symmetry $\cnos(f(t))$ could be half integral.

Let $\PP_{(n+r)/2}(t)$ denote the vector space of palindromic 
univariate polynomials $f(t) = \sum_{i=0}^n a_i t^i$ with $r$ being 
the least non-negative integer such that $a_r > 0$ and with 
$\cnos(f(t)) = (n+r)/2$.  Clearly, 
$\displaystyle \Gamma = \{ t^{r+i}(1+t)^{n-r-2i}: 0 \leq i \leq \nmrhalf \}$
is a basis of $\PP_{(n+r)/2}(t)$.  Thus, if $f(t) \in \PP_{(n+r)/2}(t)$,
then we can write
$f(t) = \sum_{i=0}^{\nmrhalf} \gamma_{n,i} t^{r+i} (1+t)^{n-2i}$.
$f(t)$ is said to be {\sl ``gamma positive"} if $\gamma_{n,i} \geq 0$ for 
all $i$ (that is, if $f(t)$ has positive coefficients when expressed in 
the basis $\Gamma$).

For a positive integer $n$, let $[n] = \{1,2,\ldots,n \}$ and let
$\SSS_n$ be the set of permutations on $[n]$.  For
$\pi = \pi_1,\pi_2,\ldots,\pi_n \in \SSS_n$, define its excedance
set as $\ExcSa(\pi) = \{ i \in [n]: \pi_i > i\}$  and its number of 
excedances as $\exc(\pi) = |\ExcSa(\pi)|$. Define its number of 
non-excedances as $\nexc(\pi) = | \{ i \in [n]: \pi_i \leq  i\}|$.
For $\pi \in \SSS_n$, define its number of inversions as $\inv(\pi) = 
| \{ 1 \leq i < j \leq n : \pi_i > \pi_j \} |$.  Let 
$\DescSet(\pi) = \{i \in [n-1]: \pi_i > \pi_{i+1} \}$ and 
$\ascSet(\pi) = \{i \in [n-1]: \pi_i < \pi_{i+1} \}$
be its set of descents and  ascents respectively. 
Let $\des(\pi) = |\DescSet(\pi)|$ be its number of descents 
and $\asc(\pi) = |\ascSet(\pi)|$  be its number of ascents. The polynomials 
$A_n(t)= \sum_{\pi \in \SSS_n} t^{\des(\pi)}$ are the classical 
Eulerian polynomials. Let $\AAA_n \subseteq \SSS_n$ be the subset of
even permutations.  Define

\vspace{-4 mm}

\begin{eqnarray}
\label{eqn:exc}
\GE_n(t) & = & \sum_{\pi \in \SSS_n} t^{\exc(\pi)} 
\mbox{ 
\hspace{ 3 mm}
and 
\hspace{ 3 mm}
}
\GE_n(s,t)  =  \sum_{\pi \in \SSS_n} t^{\exc(\pi)} s^{\antiexc(\pi)-1}, \\
\label{eqn:exc_even}
\GE_n^{+}(t) & = & \sum_{\pi \in \AAA_n} t^{\exc(\pi)} 
\mbox{ 
\hspace{ 3 mm}
and 
\hspace{ 3 mm}
}
\GE_n^{+}(s,t)  =  \sum_{\pi \in \AAA_n} t^{\exc(\pi)} s^{\antiexc(\pi)-1}, \\
\label{eqn:exc_odd}
\GE_n^{-}(t) & = & \sum_{\pi \in \SSS_n -
	\AAA_n} t^{\exc(\pi)}
\mbox{ 
	\hspace{ 3 mm}
	and 
	\hspace{ 3 mm}}
\GE_n^{-}(s,t)  =  \sum_{\pi \in \SSS_n - \AAA_n}
t^{\exc(\pi)}s^{\antiexc(\pi)-1}.
\end{eqnarray}

It is a well known result of MacMahon \cite{macmahon-book} that both 
descents and excedances are equidistributed over $\SSS_n$. 
That is, for all non negative integers $n$, $A_n(t)= \GE_n(t)$.
It is well known that the Eulerian polynomials $\GE_n(t)$ are  
palindromic (see Graham, Knuth and Patashnik \cite{graham-knuth-patashnik}).
Gamma positivity of $\GE_n(t)$  was first proved by 
Foata and Sch{\"u}tzenberger (see \cite{foata-schutzenberger-eulerian}).
Foata and Strehl \cite{foata-strehl} later used a group action
based proof which has been termed as ``valley hopping" by
Shapiro, Woan and Getu \cite{shapiro-woan-getu_runs_slides_moments}.
This approach gives a combinatorial interpretation for the gamma 
coefficients.
Several refinements of the gamma positivity of $\GE_n(t)$
are known when enumeration is done with respect to both descents 
$\des()$ and with respect to excedance $\exc()$. 
For two statistics $\twoonethree, 
\threeonetwo: \SSS_n \rightarrow \ZZ_{\geq 0}$,
Br{\"a}nd{\'e}n \cite{branden-actions_on_perms_unimodality_descents}
and Shin-and-Zeng 
\cite{shin-zeng-eulerian-continued-fraction,shin-zeng_symmetric_unimodal_excedances_colored_perms} 
have shown a $p,q$-refinement by showing  that 
the polynomial

\vspace{-6 mm}

$$A_n(p,q,t) = \sum_{\pi \in \SSS_n} p^{\twoonethree(\pi)}q^{\threeonetwo(\pi)}t^{\des(\pi)}
= \sum_{i=0}^{\nhalf} a_{n,i}(p,q) t^i(1+t)^{n-2i}$$
where the $a_{n,i}(p,q)$'s are polynomials with positive coefficients.  Similarly, Shareshian
and Wachs \cite{shareshian-wachs-gamma-positivity-variations} have shown for statistics
$\des^*, \maj: \SSS_n \mapsto \ZZ_{\geq 0}$ that the polynomial 

\vspace{-3 mm}

$$\GE_n(p,q,t) = \sum_{\pi \in \SSS_n} p^{\des^*(\pi)}q^{\maj(\pi) - \exc(\pi)}t^{\exc(\pi)}
= \sum_{i=0}^{\nhalf} \gamma_{n,i}(p,q) t^i(1+t)^{n-2i}$$
where the $\gamma_{n,i}(p,q)$'s are polynomials with positive coefficients.  Dey and 
Sivasubramanian had recently given gamma positivity results (see 
\cite{siva-dey-gamma_positive_descents_alt_group}) when one sums descents over
$\AAA_n$.  In this paper, we consider the case when we sum excedances over
$\AAA_n$.  Two of our main results in this paper are the following.

\begin{theorem}
\label{thm:maintheoremforexccedance}
For all positive integers $n \geq 5$ with $n \equiv 1$ (mod 2),  
$\GE_n^{+}(s,t)$ and $\GE_n^{-}(s,t)$ 
are gamma positive, with both polynomials having the 
same center of symmetry $(n-1)/2$.   
\end{theorem}

\begin{theorem} 
\label{thm:sumof2forevenAtype}
For all even positive integers $n=2m$ with $n \geq 4$, 
$\GE_n^+(t)$ and $\GE_n^-(t)$ can be written as a 
sum of two gamma positive polynomials
whose centers of symmetry differ by one.
\end{theorem}

Similar results are known when we sum excedances over derangements.  Let
$\SD_n = \{ \pi \in \SSS_n: \pi_i \not= i \mbox{ for all } i\}$ denote the 
set of derangements in $\SSS_n$.  Let $\SD_n^+ = \SD_n \cap \AAA_n$ and 
$\SD_n^- = \SD_n \cap (\SSS_n - \AAA_n)$. Define

\vspace{-4 mm}

\begin{eqnarray}
\label{eqn:der_exc}
\ADE_n(t) & = & \sum_{\pi \in \SD_n} t^{\exc(\pi)} 
\mbox{ 
\hspace{ 3 mm}
and 
\hspace{ 3 mm}
}
\ADE_n(s,t)  =  \sum_{\pi \in \SD_n} t^{\exc(\pi)} s^{\antiexc(\pi)-1} \\
\label{eqn:der_exc_even}
\ADE_n^{+}(t) & = & \sum_{\pi \in \SD_n^+} t^{\exc(\pi)} 
\mbox{ 
\hspace{ 3 mm}
and 
\hspace{ 3 mm}
}
\ADE_n^{+}(s,t)  =  \sum_{\pi \in \SD_n^+} t^{\exc(\pi)} s^{\antiexc(\pi)-1} \\
\label{eqn:der_exc_odd}
\ADE_n^{-}(t) & = & \sum_{\pi \in \SD_n^-} t^{\exc(\pi)}
\mbox{ 
	\hspace{ 3 mm}
	and 
	\hspace{ 3 mm}}
\ADE_n^{-}(s,t)  =  \sum_{\pi \in \SD_n^-} t^{\exc(\pi)}s^{\antiexc(\pi)-1}.
\end{eqnarray}

The polynomial $\ADE_n(t)$ is known to be real rooted (see Zhang \cite{zhang-q-derang-poly})
and easily seen to be palindromic.  Hence it is gamma positive (see 
Petersens book \cite[Page 82]{petersen-eulerian-nos-book}).  Shin and Zeng 
in \cite{shin-zeng-eulerian-continued-fraction,shin-zeng_symmetric_unimodal_excedances_colored_perms}
have proved 
the following refinement with respect to the number of inversions $\inv(\pi)$, number
of cycles $\cyc(\pi)$ and the nesting number $\nest(\pi)$.

\begin{theorem}[Shin and Zeng]
\label{thm:shin_zeng_excedance}
For all positive integers $n$, and for all choices $\stat(\pi) \in \{\cyc(\pi), \inv(\pi), \nest(\pi) \}$,
the polynomial 
$\sum_{\pi \in \SD_n} q^{\stat(\pi)}t^{\exc(\pi)} = \sum_{i=0}^{\nhalf} b_{n,i}(q) t^i(1+t)^{n-2i}$
where $b_{n,i}(q)$ is a polynomial with positive integral coefficients.
\end{theorem}

We prove the following refinement of Theorem \ref{thm:shin_zeng_excedance} in this paper.

\begin{theorem}
\label{thm:main_exc}
For all positive integers $n$, and for all choices $\stat(\pi) \in \{\cyc(\pi), \inv(\pi)\}$,
the polynomials 
$\sum_{\pi \in \SD_n^+} q^{\stat(\pi)}t^{\exc(\pi)} = \sum_{i=0}^{\nhalf} b_{n,i}^+(q) t^i(1+t)^{n-2i}$
and 
$\sum_{\pi \in \SD_n^-} q^{\stat(\pi)}t^{\exc(\pi)} = \sum_{i=0}^{\nhalf} b_{n,i}^-(q) t^i(1+t)^{n-2i}$
where both $b_{n,i}^+(q)$ and $b_{n,i}^-(q)$ 
are polynomials with positive integral coefficients.
\end{theorem}

We generalize our results to the case when excedances 
are summed up over the elements with positive sign in classical Weyl groups. 
Let $\BB_n$ denote the group of signed permutations on 
$T_n = \{-n, -(n-1), \ldots,-1,1,2,\ldots,n \}$, that is 
$\sigma \in \BB_n$ consists of all permutations of $T_n$ that satisfy 
$\sigma(-i) = -\sigma(i)$
for all $i \in [n]$.  Similarly, let $\DD_n\subseteq \BB_n$ denote the
subset consisting of those elements of $\BB_n$ which have an even
number of negative entries.  
As both $\BB_n$ and $\DD_n$ are Coxeter 
groups, they have a natural notion of descent associated to them.  
Similar to the classical Eulerian polynomials, 
we thus have the Eulerian polynomials of type-B and type-D. 
There is a natural
notion of length in these groups and we get results when restricted
to elements with even length.  
Our results for Type-B Weyl groups are  Theorems 
\ref{thm:main_thm_for_Btype_excedance} and 
\ref{thm:sumof2foroddBtypeexcedance}.  
Similarly for Type-D Weyl groups, our main result is Theorem 
 \ref{thm:d_type_exccedance_even_integers}.

\section{Preliminaries on Gamma positive polynomials}
\label{subsec:gamma_prelims}

Most of the gamma positive polynomials in this work have homogeneous 
bivariate counterparts with the following  slightly more general
definition of gamma positivity. 
Let $f(s,t) = \sum_{i=0}^n a_i s^{n-i}t^i$ be 
a homogenous bivariate polynomial with $a_n \not= 0$ and 
let $r$ be the smallest non-negative integer such that $a_r \not= 0$.  
As before, define$f(s,t)$ to be palindromic if $a_{r+i} = a_{n-i}$ 
for all $i$ and define $\cnos(f(s,t)) = (n+r)/2$.  
	
Let $f(s,t)$ have $\cnos(f(s,t))=n/2$.  
$f(s,t)$ is said to be bivariate gamma positive if 
$f(s,t) = \sum_{i=0}^{n/2} \gamma_{n,i} (st)^i(s+t)^{n-2i}$ with $\gamma_{n,i}
\geq 0$ for all $i$.  Clearly, if $f(s,t)$ is 
a bivariate gamma positive polynomial, then $f(t) = f(s,t)|_{s=1}$ is 
clearly a univariate gamma positive palindromic polynomial with the same
center of symmetry.  We need the following lemmas.  Since all of 
them  are from  
\cite{siva-dey-gamma_positive_descents_alt_group}), we omit proofs.
	
\begin{lemma}
   \label{lem:prod_bivariate_gamma_nonneg}
		Let $f_1(s,t)$ and $f_2(s,t)$ be two bivariate gamma positive 
		polynomials with respective centers of symmetry 
		$\cnos(f_1(s,t)) = r_1 + (n-r_1)/2$ and 
		$\cnos(f_2(s,t)) = r_2 + (n-r_2)/2$.
		Then $f_1(s,t) f_2(s,t)$ is gamma positive with 
		$\cnos(f_1(s,t)f_2(s,t)) =  (r_1+r_2) + (2n-r_1-r_2)/2$.
\end{lemma}

	Let $D$ be the operator $\displaystyle \left( \frac{d}{ds} + \frac{d}{dt}
	\right)$ in $\QQ[s,t]$.  
	
	\begin{lemma}
		\label{lem:derivative_gamma_nonneg}
		Let $f(s,t)$ be a bivariate gamma positive polynomial with 
		$\cnos(f(s,t)) = n/2$.  Then, $g(s,t) = D f(s,t)$
		is gamma positive with $\cnos(g(s,t)) = (n-1)/2$.
	\end{lemma}
	
	The following are easy corollaries.
	
	\begin{corollary}
		\label{cor:gamma_nonneg}
		Let $f(s,t)$ be a bivariate gamma positive polynomial with 
		$\cnos(f(s,t)) = n/2$.  
		\begin{enumerate}
			\item Then, for a natural number $\ell$, 
			$\displaystyle g(s,t) = D^{\ell} f(s,t)$ is 
			gamma positive with $\cnos(g(s,t)) = (n-\ell)/2$.
			\item Then, $h(s,t) = (st)^i f(s,t)$ is gamma positive with 
			$\cnos( h(s,t) ) = i+ n/2$
			\item Then, $h(s,t) = (s+t) f(s,t)$ is gamma positive with
			$\cnos( h(s,t) ) = (n+1)/2$
		\end{enumerate}
	\end{corollary}

	\begin{lemma}
		\label{lem:more_gamma_positive}
		Let $f(t)$ be gamma positive with $\cnos(f(t)) = n/2$ and with
		$\lngt(f(t))$ being odd.  Then, $f(t)$ is the sum of 
		two gamma positive sequences $p_1(t)$ and $p_2(t)$ with centers 
		of symmetry $(n-1)/2$ and $(n+1)/2$ respectively.  Further,
		both $p_1(t)$ and $p_2(t)$ have even length.
	\end{lemma}

\section{Type-A Coxeter Groups}

Define $\np(\pi)$=index $i$ such that $\pi(i)=n$ and  define $\pi'$ to be $\pi$ 
restricted to $[n-1]$.  Let
$\SSS_n^i=\{\pi \in \SSS_n: \np(\pi)=i\}$, 
$\AAA_n^i=\{\pi \in \AAA_n: \np(\pi)=i\}$, and
$\SSS_n^i-\AAA_n^i=\{\pi \in \SSS_n-\AAA_n: \np(\pi)=i\}$.  We need Foata's First Fundamental 
Transformation which is a bijection that maps descents to excedances 
(See Lothaire \cite[Section 10.2]{lothaire-combin-words}).

\begin{theorem}[Foata's First Fundamental Transformation] 
\label{thm:Foatafirstfunda}
For all $n \geq 2$, there exist a bijection $\FFT_n:\SSS_n \mapsto \SSS_n$  such that 
$\des(\FFT_n(\pi))= \exc(\pi)$.  Hence, the statistics $\des()$ and $\exc()$ are equidistributed over 
$\SSS_n$ for all positive integers $n$. 
\end{theorem}

We use the bijection $\FFT_n$ of Theorem \ref{thm:Foatafirstfunda} to prove the following.  
    
\begin{lemma}
\label{lem:descentinitialexcedancepenultimate}
For all integers $n \geq 2$, the following holds: 
\begin{equation}\label{eqn:equationdescetinitialexcedancepenultimate}
\sum_{\pi \in \SSS_n^1}
t^{\des(\pi)}s^{\asc(\pi)}=
\sum_{\pi \in \SSS_{n}^{n-1}}t^{\exc(\pi)}
s^{\antiexc(\pi)-1}
\end{equation}
\end{lemma}

\begin{proof}
We give a  bijection  $f_{n} : \SSS_n^{n-1} \mapsto\SSS_n^{1}$ as follows. 
For any $\pi= \pi_1,\pi_2,\ldots, \pi_{n-2},n,\pi_{n}  \in \SSS_n^{n-1}$, 
let $\FFT_{n-1}(\pi') = a_1, a_2, \ldots, a_{n-1} $.  Then, let 
$f_n(\pi)= n,a_1, a_2, \ldots,a_{n-1}$. 
Clearly, $f_n$ is well-defined and is a bijection.  Further, for 
$\pi \in \SSS_n^{n-1}$, it is 
easy to see that $\exc(\pi) = \des(f_n(\pi))$ and hence
that $\nexc(\pi) = \asc(f_n(\pi)) + 1$.
\end{proof}

\begin{lemma}
\label{lem:excedanceallexceptpenultimateandlast}
For all positive integers $n$ and $r$ with $1
\leq r \leq n-2$,
the following holds:
\begin{eqnarray*} 
\sum_{\pi \in \AAA_n^r } t^{\exc(\pi)} s^{\antiexc(\pi)-1}=
\sum_{\pi \in \SSS_n^r-\AAA_n^r} t^{\exc(\pi)} s^{\antiexc(\pi)-1}=
\frac{1}{2}{\sum_{\pi \in \SSS_n^r}} t^{\exc(\pi)} s^{\antiexc(\pi)-1}
\end{eqnarray*}
\end{lemma}
    
\begin{proof}
Define $g: \AAA_n^r
\mapsto \SSS_n^r-\AAA_n^r $ by $g(\pi_1,\pi_2,\ldots,\pi_{n-1},\pi_{n})=
\pi_1,\pi_2,\ldots,\pi_{n},\pi_{n-1}$.  Clearly, $\inv(\pi) \not \equiv \inv(g(\pi))$ (mod 2) 
and further $g^2=\id$.  As 
$\np(\pi) \leq n-2$,  we also have $\exc(\pi) = \exc(g(\pi))$, completing the proof.
\end{proof}
    

\begin{lemma}
\label{lem:recurrenceforAEXc_n^+}
For all positive integers $n$, the polynomials
$\GE_n^{+}(s,t)$ and $\GE_n^{-}(s,t)$ satisfy the 
following recurrence.
\begin{eqnarray}
\label{eqn:recurrenceforAexc_neven}
\GE_n^{+}(s,t)& = & 
s\GE_{n-1}^{+}(s,t) +t\GE_{n-1}^{-}(s,t)+ \frac{1}{2}st DA_{n-1}(s,t),\\ 
\label{eqn:recurrenceforAexc_odd}	   	
\GE_n^{-}(s,t)& = & 
s\GE_{n-1}^{-}(s,t)+ t\GE_{n-1}^{+}(s,t)+ \frac{1}{2}st D A_{n-1}(s,t).
\end{eqnarray}
\end{lemma}

\begin{proof}
Foata's First Fundamental bijection, applied on 
$\SSS_n$ and $\SSS_{n-1}$  gives us: 
\begin{eqnarray}
\label{eqn:equationFoatabijectiononSn}
\sum_{\pi \in \SSS_n }t^{\des(\pi)}s^{\asc(\pi)} & = & \sum_{\pi \in \SSS_n }t^{\exc(\pi)}s^{\antiexc(\pi)-1} \\
\label{eqn:equationFoatabijectiononSn-1}
\sum_{\pi \in \SSS_n^{n}}t^{\des(\pi)} s^{\asc(\pi)} & = & \sum_{\pi \in \SSS_n^{n}}t^{\exc(\pi)} 
s^{\antiexc(\pi)-1}
\end{eqnarray}

    	Subtracting equations
        \eqref{eqn:equationdescetinitialexcedancepenultimate}
    	 and \eqref{eqn:equationFoatabijectiononSn-1} from 
    	 \eqref{eqn:equationFoatabijectiononSn}, we get
    	 \begin{equation}
    	 \label{eqn:requidistribution_over_restricted_Sets}
    	\sum_{r=2}^{n-1} \sum_{\pi \in \SSS_n^r}t^{\des(\pi)}
    	s^{\asc(\pi)}=
    	\sum_{r=1}^{n-2} \sum_{\pi \in \SSS_n^r}t^{\exc(\pi)}
    	s^{\antiexc(\pi)-1}
    	\end{equation}
    	
 We know 
\begin{eqnarray*}
\GE_n^{+}(s,t) & = &\sum_{r=1}^{n-2}\sum_{\pi \in \AAA_n^r }t^{\exc(\pi)}s^{\antiexc(\pi)-1}
    	 + \sum_{\pi \in \AAA_n^{n-1} }t^{\exc(\pi)}s^{\antiexc(\pi)-1}+
    	 \sum_{\pi \in \AAA_n^n }t^{\exc(\pi)}s^{\antiexc(\pi)-1}\\
& = & \frac{1}{2}\sum_{r=1}^{n-2}\sum_{\pi \in \SSS_n^r }t^{\exc(\pi)}s^{\antiexc(\pi)-1}+t\GE_{n-1}^{-}(s,t)+
 s\GE_{n-1}^{+}(s,t) 
\\ 
& = & \frac{1}{2}\sum_{r=2}^{n-1} \sum_{\pi \in \SSS_n^r}t^{\des(\pi)} s^{\asc(\pi)}+  
	t\GE_{n-1}^{-}(s,t)+ s\GE_{n-1}^{+}(s,t)\\
& = & \frac{1}{2} stDA_{n-1}(s,t)+ t\GE_{n-1}^{-}(s,t)+ s\GE_{n-1}^{+}(s,t)
\end{eqnarray*} 
The second line follows by using Lemma \ref{lem:excedanceallexceptpenultimateandlast}. The third line 
follows using \eqref{eqn:requidistribution_over_restricted_Sets}. 
The last line follows using the fact that $stDA_{n-1}(s,t)$ is 
the contribution of all the permutations where $n$ is not in the 
initial or the final position. The proof of 
\eqref{eqn:recurrenceforAexc_odd} is identical and is hence omitted.
\end{proof}

A simple corollary of Lemma \ref{lem:recurrenceforAEXc_n^+} is 
the following Theorem of Mantaci (see \cite{mantaci-jcta-93}). 

\begin{corollary}
\label{cor:mantaci_theorem}
For all positive integers $n\geq 2$, 
$\sum_{\pi \in {\SSS_n} }(-1)^{\inv(\pi)} t^{\exc(\pi)} s ^{\antiexc(\pi)-1}  
=  (s-t)^{n-1}.$
\end{corollary}  
\begin{proof}
We induct on $n$ with the base case being $n=2$.
When $n=2$, it is simple to note that
$\sum_{\pi \in {\SSS_2} }(-1)^{\inv(\pi)} t^{\exc(\pi)} s ^{\antiexc(\pi)-1}  
= s-t$.  
Subtracting \eqref{eqn:recurrenceforAexc_odd} from  \eqref{eqn:recurrenceforAexc_neven} we get
\begin{eqnarray}
  \label{eqn:signed_inversion_recurrence}
  \sum_{\pi \in \SSS_n }(-1)^{\inv(\pi)}
  t^{\exc(\pi)}	s^{\antiexc(\pi)-1} & = &\GE_n^{+}(s,t)- \GE_n^{-}(s,t) \nonumber \\ 
  & = & (s-t)\left[\GE_{n-1}^{+}(s,t)- \GE_{n-1}^{-}(s,t)\right] \nonumber 
   =  (s-t)^{n-1}
  \end{eqnarray}
\end{proof}  

Let $\GE_n^+(t) = \sum_{k=0}^{n-1} a_{n,k}^+ t^k$ and 
let $\GE_n^-(t) = \sum_{k=0}^{n-1} a_{n,k}^- t^k$.
The following recurrences for the numbers $a_{n,k}^+$ and 
$a_{n,k}^-$ were proved by Mantaci in \cite{mantaci-jcta-93, mantaci-thesis}.
It is easy to derive them from Lemma \ref{lem:recurrenceforAEXc_n^+} 

\begin{corollary}
\label{cor:recurrencesforexcedancecoefficients}
For all positive integers $n$ with $n \geq 2$,  we have 
\begin{eqnarray*}
a_{n,k}^+ = ka_{n-1,k}^- + (n-k)a_{n-1,k-1}^- + a_{n-1,k}^+
\mbox{ \hspace{1 mm} and \hspace{1 mm}}
a_{n,k}^- = ka_{n-1,k}^+ + (n-k)a_{n-1,k-1}^+ + a_{n-1,k}^-.
\end{eqnarray*}
\end{corollary}
	
We move on to palindromicity of the polynomials 
$\GE_n^{+}(s,t)$ and $\GE_n^{-}(s,t)$.

\begin{lemma}
\label{lem:palindromic} 
For all natural numbers $n$, $\GE_n^{+}(s,t)$ and $\GE_n^{-}(s,t)$ are 
palindromic if and only if  $n \equiv 1 \mod2$.
\end{lemma}
\begin{proof} 
By Corollary \ref{cor:mantaci_theorem}, 
$$\GE_n^{+}(s,t) =  \frac{1}{2} \left( 
A_{n}(s,t)+ (s-t)^{n-1} \right) \mbox{ \hspace{2 mm} and \hspace{2 mm} }
\GE_n^{-}(s,t) =  \frac{1}{2} \left( 
A_{n}(s,t)- (s-t)^{n-1} \right).$$
  	
Clearly, $A_n(s,t)$ is palindromic for all $n$ while $(s-t)^{n-1}$ 
is palindromic if and only if $n$ is odd. 
Hence, $\GE_n^{+}(s,t)$ and $\GE_n^{-}(s,t)$ 
are palindromic iff  $n \equiv 1 \mod2$. 
\end{proof}

\begin{lemma}\label{lem:excedanceplusequalshalfexcedance}
For all positive integers $n \geq 2$,	
$D \GE_{n}^{+}(s,t)=D\GE_{n}^{-}(s,t)=
    	\frac{1}{2} D A_{n}(s,t).$
\end{lemma}
\begin{proof}
We induct on $n$.  Clearly, when $n=2$, 
$\GE_{2}^{+}(s,t) = s$, $\GE_{2}^{-}(s,t)=t$ while $A_2(s,t) = s + t$.
Thus, $D\GE_{2}^{+}(s,t)=D\GE_{2}^{-}(s,t)=1=\frac{1}{2}DA_{2}(s,t.$ 
Assume, $D
\GE_{n-1}^{+}(s,t)=D\GE_{n-1}^{-}(s,t)=
\frac{1}{2}DA_{n-1}(s,t).$  Further,
    	 
\begin{eqnarray*}
D\GE_{n}^{+}(s,t) 
	& = &	D \left(s\GE_{n-1}^{+}(s,t) +t\GE_{n-1}^{-}(s,t)+ \frac{1}{2} stDA_{n-1}(s,t)\right) \\
    	& = & sD\GE_{n-1}^{+}(s,t)+ \GE_{n-1}^{+}(s,t) + tD\GE_{n-1}^{-}(s,t)+ \GE_{n-1}^{-}(s,t) \\
    	&  &  + \frac{1}{2} D (stDA_{n-1}(s,t)) \\
    	& = & sD\GE_{n-1}^{-}(s,t)+ \GE_{n-1}^{+}(s,t) + tD\GE_{n-1}^{+}(s,t)+ \GE_{n-1}^{-}(s,t) \\
    	&  & + \frac{1}{2} D (st DA_{n-1}(s,t) \\
    	& = & D \left(s\GE_{n-1}^{-}(s,t) +t\GE_{n-1}^{+}(s,t)+ \frac{1}{2} stDA_{n-1}(s,t)\right) = D\GE_{n}^{-}(s,t)
    	\end{eqnarray*}
The proof is complete. 
\end{proof}

We use the recurrences \eqref{eqn:recurrenceforAexc_neven} and 
\eqref{eqn:recurrenceforAexc_odd}
four times to get the following.
Since this is the main idea of the proof, we omit some details.

\begin{lemma}
\label{lem:induct_by_4}
For all positive integers $n$, we have
\begin{eqnarray}
\GE_{n+4}^+(s,t) 
	& = & L_1(s,t) \GE_{n}^+(s,t) + L_2(s,t) \GE_{n}^-(s,t) + L_3(s,t) D \GE_{n}^+(s,t) \nonumber \\ 
	& & + \left[ L_4(s,t) D^2  + L_5(s,t) D^3 + L_6(s,t) D^4 \right] \GE_{n}^{+}(s,t) \label{eqn:n+4_jump}  \\
\GE_{n+4}^-(s,t) & = & L_1(s,t) \GE_{n}^-(s,t) + L_2(s,t) \GE_{n}^+(s,t) + L_3(s,t) D \GE_{n}^-(s,t) \nonumber \\ 
	& & + \left[ L_4(s,t) D^2 + L_5(s,t) D^3 + L_6(s,t) D^4 \right] \GE_{n}^{-}(s,t)  \label{eqn:n+4_jump2} 
\end{eqnarray}
where the $L_i(s,t)$ with their centers of symmetry are as follows.
	
$
\begin{array}{l|r}
f(s,t) & \cnos(f(s,t)) \\ \hline
L_1(s,t) = (s+t)^4+7st(s+t)^2+16(st)^2 &  2 \\ \hline
L_2(s,t) = 15st(s+t)^2  &  2 \\ \hline 
L_3(s,t) = 3(5s^2+30st+5t^2)st(s+t) &  5/ 2 \\ \hline
L_4(s,t) = 25(st)^2(s+t)^2+20(st)^3&  3\\ \hline
L_5(s,t) = 10(st)^{3}(s+t) & 7/2 \\ \hline
L_6(s,t) = (st)^4 & 4 
\end{array}
$
\end{lemma}

\begin{proof} (Of Theorem \ref{thm:maintheoremforexccedance})
When $n=5$  and $n=7$ one can check that 
\begin{eqnarray*}
\GE_{5}^{+}(s,t) & = & s^4+11s^3t+36s^2t^2+11st^3+t^4  
=  (s+t)^4+7st(s+t)^2+16(st)^2, \\
\GE_{5}^{-}(s,t) & = &15s^3t+30s^2t^2+15st^3
= 15st(s+t)^2, \label{eqn:m+4_jump}\\
\GE_{7}^{+}(s,t) & = &  s^6+57s^5t+603s^4t^2+
1198s^3t^3+603s^2t^4+57st^5+t^6 \\ 
& = &(s+t)^6+51st(s+t)^4+384(st)(s+t)^2+104(st)^3 , \\ 
\GE_{7}^{-}(s,t) & = & 63s^5t+588s^4t^2+1218s^3t^3+
588s^2t^4+63st^5 , \\
& = & 63st(s+t)^4+336(st)^2(s+t)^2+168(st)^3. 
\end{eqnarray*}
Let $n$ be odd with $n >7$ and let $m = n-4$.  By induction, 
both $\GE_m^{+}(s,t)$ and $\GE_m^{-}(s,t)$ are 
gamma positive with centers of symmetry $\frac{1}{2}(m-1) $.
It is easy to check that all $L_i(s,t)$'s for all $1 \leq i \leq 6$
that appear in Lemma \ref{lem:induct_by_4} are gamma positive. 
Further, each of the six terms in \eqref{eqn:n+4_jump} has the same 
center of symmetry $\frac{1}{2}(m+3)$.  Thus by
Lemma \ref{lem:induct_by_4}, $\GE_n^+(s,t)$ is gamma positive.
In an identical manner, one can prove that 
$\GE_{n}^-(s,t)$ is gamma positive with 
center of symmetry $\frac{1}{2} (n-1)$, completing the proof.
\end{proof}

\subsection{When $n \equiv 0$  (mod 2)} 
When $n=2m$ for a natural number $m$, by Lemma \ref{lem:palindromic}, 
the polynomials $\GE_{n}^+(t)$ and $\GE_{n}^-(t)$ are not even palindromic.

\begin{proof} (Of Theorem \ref{thm:sumof2forevenAtype})
We induct on $n$.  Our base case is when $n=2m=4$, one can 
check that $\GE_4^{+}(t)= 1+4t+7t^2=(1+4t+t^2)+6t^2$ and $\GE_4^-(t)=
t^3+4t^2+7t=t(1+4t+t^2)+6t$.  
By Lemma \ref{lem:recurrenceforAEXc_n^+}, we have
\begin{eqnarray}
\GE_{2m+2}^+(s,t)  & =  & s \GE_{2m+1}^+(s,t) + t \GE_{2m+1}^-(s,t) 
+ st D \GE_{2m+1}^+(s,t) \nonumber \\
\label{eqn:n_even_pi_even_recurrence_univariate}
\GE_{2m+2}^+(t)  & =  &  \GE_{2m+1}^+(t) + t \GE_{2m+1}^-(t) 
+ \left( st D \GE_{2m+1}^+(s,t) \right)|_{s=1} 
\end{eqnarray}
			
The polynomial $p(s,t) = st D \GE_{2m+1}(s,t)$ is gamma positive with 
$\cnos(p(s,t)) = m-1/2+1$.  Hence, $p(t) = p(s,t)|_{s=1}$ is univariate 
gamma positive polynomial with  $\cnos(p(t)) = m+1/2$.  Further, it is 
simple to see that $p(t)$ has odd length and thus by Lemma 
\ref{lem:more_gamma_positive},  it can be written as $p(t) = p_1(t) + p_2(t)$ 
with $\cnos(p_1(t)) = m$ and $\cnos(p_2(t)) = m+1$.  Thus, 
\eqref{eqn:n_even_pi_even_recurrence_univariate}  becomes
\begin{eqnarray}
\GE_{2m+2}^+(t) &=&  \GE_{2m+1}^+(t) + t \GE_{2m+1}^-(t)
+ p_1(t) + p_2(t) \\
&=  & w_1(t) + w_2(t) 
\end{eqnarray}
where $w_1(t) = \GE_{2m+1}^+(t) + p_1(t)$ has $\cnos(w_1(t)) = m$ 
and $w_2(t) = t\GE_{2m+1}^-(t) + p_2(t)$ has $\cnos(w_2(t)) = m+1$.
The proof is complete.
\end{proof}

\section{Type B Coxeter Group}
    
Let $\BB_n$ be the set of permutations $\pi$ of $\{-n, -(n-1),
\ldots, -1, 1, 2, \ldots n\}$ 
satisfying $\pi(-i) = -\pi(i)$.  We denote $-k$ as $\ob{k}$ as well.
$\BB_n$ is referred to as the hyperoctahedral group or the group of 
signed permutations on $[n]$ and $|\BB_n| = 2^n n!$. 
For $\pi \in \BB_n$ we alternatively denote $\pi(i)$ as $\pi_i$. 
For $\pi \in \BB_n$, define $\Negs(\pi) = 
    \{\pi_i : i > 0, \pi_i < 0 \}$ be 
    the set of elements which occur with a negative sign.  
    Define $\inv_B(\pi)=  | \{ 1 \leq i < j \leq n : \pi_i > \pi_j \} |
    + \sum_{i \in \Negs(\pi)} i$ and
let $\BB^+_n \subseteq \BB_n$ denote the subset of 
elements having even $\inv_B()$ value and let $\BB_n^- = \BB_n - \BB_n^+$.

We follow Brenti's\cite{brenti-q-eulerian-94} definition of excedance: 
$\exc_B(\pi)= |\{ i \in [n]: \pi_{|\pi(i)|} > \pi_ i\}|+|\{ i \in [n]: \pi_i =- i\}|$ and
    $\antiexc_B(\pi)=n-\exc_B(\pi)$. 
    For all the permutations $\pi \in \BB_n $, let $ \pi_0=0$. 
    Define $\des_B(\pi)= 
    |\{ i \in [0,1,2\ldots ,n-1]: \pi_i > \pi_{i+1}\}| $ and 
    $\asc_B(\pi)= |\{ i \in [0,1,2\ldots ,n-1]: \pi_i <  \pi_{i+1}\}|$. 
    	Define

\begin{eqnarray}
\label{eqn:Bn} 
\EE_n(t) & = & \sum_{\pi \in \BB_n} t^{\exc_B(\pi)} 
	\mbox{\hspace{3 mm} and \hspace {3 mm} }
\EE_n(s,t)  =  \sum_{\pi \in \BB_n}t^{\exc_B(\pi)}s^{\antiexc_B(\pi)} \\ 
\label{eqn:descentBtypebivariate}
B_n(t) & = & \sum_{\pi \in \BB_n} t^{\des_B(\pi)}
	\mbox{\hspace{3 mm} and \hspace {3 mm} }
B_n(s,t)  =  \sum_{\pi \in \BB_n} t^{\des_B(\pi)} s^{\asc_B(\pi)}    \\  
\label{eqn:descentBtypeeven}
B_n^+(s,t) & = & \sum_{\pi \in \BB_n^+} t^{\des_B(\pi)} s^{\asc_B(\pi)}
\mbox{\hspace{3 mm} and \hspace {3 mm} } 
B_n^-(s,t)  =  \sum_{\pi \in \BB_n^-} t^{\des_B(\pi)} s^{\asc_B(\pi)} \\
\label{eqn:BAn} 
\EE_n^+(t) & = & \sum_{\pi \in \BB_n^+} t^{\exc_B(\pi)}  
\mbox{\hspace{3 mm} and \hspace {3 mm} }
\EE_n^+(s,t)  = \sum_{\pi \in \BB_n^+} t^{\exc_B(\pi)} s^{\antiexc_B(\pi)}\\  
\label{eqn:BnAn} 
\EE_n^-(t) & = & \sum_{\pi \in \BB_n^-} t^{\exc_B(\pi)} 
\mbox{\hspace{3 mm} and \hspace {3 mm} } 
\EE_n^-(s,t)  =  \sum_{\pi \in \BB_n^-} t^{\exc_B(\pi)} s^{\antiexc_B(\pi)}\\
\EEsgnE_n(s,t) & = & \sum_{\pi \in \BB_n} (-1)^{\inv_B(\pi)} t^{\exc_B(\pi)} 
s^{\antiexc_B(\pi)} \\
\EEsgnE_n(t)  & =  & \sum_{\pi \in \BB_n} (-1)^{\inv_B(\pi)} t^{\exc_B(\pi)} 
\end{eqnarray}

$B_n(s,t)$ is called the type-B Eulerian polynomial and 
Brenti in  \cite{brenti-q-eulerian-94} proved that $ B_n(s,t) = \EE_n(s,t)$. 

\begin{theorem}[Brenti]
\label{thm:brenti_fft}
 For all positive integers $n$, there exists a bijection $h_n:\BB_n \mapsto \BB_n$ 
 such that $\asc_B(h_n(\pi))=\wexc_B(\pi)$ and $\Negs((h(_n\pi))= \Negs(\pi)$. 
\end{theorem}

Gamma positivity of the type-B Eulerian polynomial was
shown by Chow \cite{chow-certain_combin_expansions_eulerian} and
Petersen \cite{petersen-enriched_p_partitions_peak_algebras}.  We refine
these results to $\BB_n^+$ and $\BB_n^-$ when $n$ is even.  
We enumerate $B_n(t)$ with type-B 
inversions and find surprisingly that the signed type-B Eulerian 
polynomial is the same irrespective of whether we use descents
or excedances (see Theorem \ref{thm:signed_descent_ascent_npos_for_B_n}).

\comment{
Dey and Sivasubramanian in \cite{siva-dey-gamma_positive_descents_alt_group}
gave the following recurrence for $B_n(s,t)$. 

\begin{theorem}
For all positive integers $n \geq 2$, the following recurrrence is true. 
\begin{equation}
\label{eqn:rec_biv_eulerian_B}
B_{n+1}(s,t) = (s+t)B_n(s,t) + 2stD B_n(s,t).
\end{equation}
\end{theorem}
       
\begin{corollary}
\label{cor:foata_schutzenberger_typeB}
For all $n \geq 2$  we have the following recurrence relation.
\begin{equation}
\label{eqn:rec_biv_eulerian_excedance_B}
\EE_{n+1}(s,t) = (s+t)\EE_n(s,t) + 2stD \EE_n(s,t).
\end{equation} 
\end{corollary} 
}       	
       	
\comment{	
\begin{proof} We observe  that the term $(s+t)B_n(s,t)$ is the contribution of all 
those permutations $\pi \in \BB_n$ where the 
letter `$n+1$' or `$\ob{n+1}$', is at the last position.  It is 
easy to see that 
$\displaystyle 2st
DB_n(s,t)$ 
is the contribution of all $\pi \in \BB_n$ in which 
the letter `$n+1$' or `$\ob{n+1}$' appears in position $r$
for $1 \leq r \leq n$.  
\end{proof}
}

 \subsection{Equidistribution of Descent and excedance over $\BB_n^+$}       
  Let $a_1,a_2,\ldots,a_n$ be $n$ distinct positive integers with 
  $a_1 < a_2 < \cdots < a_n$. Let $\BB_{\{a_1,a_2,\ldots,a_n \}}$ 
  be the Type-B group of $2^nn!$ permutations on the letters 
  $a_1,a_2,\ldots,a_n$.  Clearly, when 
  $a_i = i$ for all $i$, we get $\BB_n = \BB_{ \{1,2,\ldots,n \}}$. 
  We write $\pi \in \BB_{\{a_1,a_2, \ldots, a_n\} }$ in two line
  notation with $a_1, a_2, \ldots, a_n$ above and $\pi_{a_i}$ below
  $a_i$.  Thus, $\pi_{a_i}$ is defined for all $i$.
  For any permutation $\pi \in  \BB_{\{a_1,a_2,\ldots,a_n \}}$
  define $\inv_B(\pi)= | \{ 1 \leq i < j \leq n : \pi_{a_i} > \pi_{a_j} \} |+  
  | \{ 1 \leq i < j \leq n : - \pi_{a_i} > \pi_{a_j} \} |+   \} +|\Negs(\pi)|$.
  Let $\BB_{ \{a_1,a_2,\ldots,a_n \} }^+$ 
  denote the subset of even length elements of $\BB_{\{a_1,a_2,\ldots,a_n \}}$ and 
  $\BB_{ \{a_1,a_2,\ldots,a_n \} }^- = \BB_{ \{a_1,a_2,\ldots,a_n \}}- \BB_{ \{a_1,a_2,\ldots,a_n \}}^+$. 
  Let $a_0 = 0$ and $\pi(0)=0$ for all $\pi$.  Define  
  $\des_B(\pi)= |\{ i \in \{0,1,2\ldots ,n-1\}: \pi_{a_i} > \pi_{a_{i+1}}\}| $ and 
  $\asc_B(\pi)= |\{ i \in \{0,1,2\ldots ,n-1\}: \pi_{a_i} <  \pi_{a_{i+1}}\}|$. 
  Also define $\pos_{a_i}(\pi)=k$ if $|\pi_{a_k}| =a_i$.  
  Define 
  \begin{eqnarray}
  \SB_{\{a_1,a_2,\ldots,a_n\}}(s,t,u)=\sum_{\pi \in \BB_{\{a_1,a_2,\ldots,a_n \}} }
  (-1)^{\inv_B(\pi)}t^{\des_B(\pi)}s^{\asc_B(\pi)}u^{\pos_{a_n}(\pi)}
  \end{eqnarray}

We recall the following result from 
\cite[Lemma 26]{siva-dey-gamma_positive_descents_alt_group}. 
  
\begin{lemma}
\label{lem:signed_sum_zero}
For all positive integers $a_1<a_2<\ldots<a_n$, 
\begin{eqnarray}
\label{eqn:an_not_at_last_position}
\sum_{\pi \in \BB_{{\{a_1,a_2,\ldots,a_{n} \}}},
\pos_{a_n}(\pi) \neq n  }(-1)^{\inv_B(\pi)}t^{\des_B(\pi)}
s^{\asc_B(\pi)} u^{\pos_{a_n}(\pi)}= 0
\end{eqnarray}
\end{lemma}

\begin{theorem} 
\label{thm:signed_descent_ascent_npos_for_B_n}
For all positive integers $a_1<a_2<\ldots<a_n$, the following holds: 
\begin{eqnarray} 
\label{eqn:signed_descent_ascent_npos_for_B_n }
\SB_{\{a_1,a_2,\ldots,a_n\}}(s,t,u)=
\sum_{\pi \in \BB_{{\{a_1,a_2,\ldots,a_n \}}} }
(-1)^{\inv_B(\pi)}t^{\des_B(\pi)}s^{\asc_B(\pi)}u^{\pos_{a_n}(\pi)}
=(s-t)^nu^n.
\end{eqnarray}
\end{theorem}
\begin{proof}
We induct on $n$. When $n=1$, clearly $\SB_{a_1}(s,t,u)=(s-t)u$.
Assume that for any $n-1$ integers $a_1<a_2<\ldots<a_{n-1}$,  
\begin{equation*}
\SB_{\{a_1,a_2,\ldots,a_{n-1}\}}(s,t,u)=
\sum_{\pi \in \BB_{{\{a_1,a_2,\ldots,a_{n-1} \}}} }(-1)^{\inv_B(\pi)}t^{\des_B(\pi)}s^{\asc_B(\pi)}
u^{\pos_{a_{n-1}}(\pi)}=(s-t)^{n-1}u^{n-1}.
\end{equation*} 
Adding $ \pm a_n$ at position $n$, we get   
\begin{eqnarray}
\label{eqn:an_at_last}
\sum_{\pi \in \BB_{{\{a_1,a_2,\ldots,a_{n-1},a_{n} \}}},
\pos_{a_{n}}(\pi)=n }(-1)^{\inv_B(\pi)}t^{\des_B(\pi)}s^{\asc_B(\pi)}
u^{\pos_{a_{n}}(\pi)}=(s-t)(s-t)^{n-1}u^{n}\
\end{eqnarray}     	
where $s(s-t)^{n-1}u^{n}$ is the contribution of 
of $\pi \in \BB_{a_1,a_2,\ldots,a_n}$ with $a_n$ appearing in the 
final position of $\pi$ and 
$-t(s-t)^{n-1}u^{n}$ is the contribution of $\pi \in 
\BB_{a_1,a_2,\ldots,a_n}$ with $\ob{a_n}$ appearing in the 
final position of $\pi$.  The proof is 
complete by summing \eqref{eqn:an_not_at_last_position} and 
\eqref{eqn:an_at_last}. 
  \end{proof}

Setting $u=1$ and $a_i=i$ in Theorem 
\ref{thm:signed_descent_ascent_npos_for_B_n}, 
we get the following:
  
\begin{corollary}  
\begin{eqnarray}
\label{eqn:signed_des_asc_over_Bn}
\SB_n(s,t)  =\sum_{\pi \in \BB_{n} }(-1)^{\inv_B(\pi)}t^{\des_B(\pi)}s^{\asc_B(\pi)}=(s-t)^n.
\end{eqnarray}
\end{corollary}

\comment{
Sivaramakrishnan in \cite{siva-sgn_exc_hyp} has proved the following theorem. 
\begin{theorem}\label{thm:formula_uniivariate_signed_excedance}
For all positive integers $n$,  
\begin{eqnarray}
\EEsgnE_{n}(t) = (1-t)^{n}
\end{eqnarray} 
\end{theorem} 
}

Sivasubramanian in \cite{siva-sgn_exc_hyp} showed that 
$\EEsgnE_n(t) = (1-t)^n$ and the same proof shows the
following.

\vspace{-5 mm}

\begin{eqnarray} 
\label{eqn:formulabivariatesignedexcedanceB}
\EEsgnE_{n}(s,t) = (s-t)^{n}
\end{eqnarray}  

We get the following surprising result about enumerating 
type-B excedances and type-B descents in $\BB_n^+$.  Note that
an analogous result is not true over $\SSS_n$.  As mentioned
earlier, Theorem \ref{thm:brenti_fft} due to Brenti 
gives a type-B bijective counterpart of 
Foata's First Fundamental Transformation.  It is 
easy to check that Brenti's bijection neither preserves 
nor reverses parity of $\inv_B()$ and so is not directly useful
to prove the result below.

\begin{theorem}
\label{thm:descentandexcedanceequidistributedoverEvenBtype}
For all positive integers $n$, type-B excedances and type-B 
descents are equidistributed over $\BB_n^+$ and $\BB_n^-$.  That is,
$$ \EE_n^+(s,t) = B_n^+(s,t) \mbox{\hspace{5 mm} and \hspace{5 mm}}   \EE_n^-(s,t) = B_n^-(s,t).$$ 
\end{theorem}
\begin{proof}
Follows from \eqref{eqn:signed_des_asc_over_Bn} and
\eqref{eqn:formulabivariatesignedexcedanceB}.
\end{proof}

In \cite{siva-dey-gamma_positive_descents_alt_group}, Dey and 
Sivasubramanian proved several results about gamma positivity of 
$B_n^+(s,t)$ and $\BB_n^-(s,t)$.  Due to 
Theorem \ref{thm:descentandexcedanceequidistributedoverEvenBtype},
these are valid for $\EE_n^+(s,t)$ and $\EE_n^-(s,t)$.  The 
following three results follow from 
\cite[Theorems 28,29,30]{siva-dey-gamma_positive_descents_alt_group}.
    
\begin{lemma}
The following recurrence relations are true  for
$\EE_n^+(s,t)$ and $\EE_n^-(s,t)$ for all positive integers
$n \geq 2$. 
\begin{eqnarray}\label{rec.Btypeexcedanceeven}
\EE_n^+(s,t) & =&  s\EE_{n-1}^+(s,t)+
t\EE_{n-1}^-(s,t)+stD
\EE_{n-1}(s,t), \vspace{-5 mm} \\\label{rec.Btypeexcedanceodd}
    	 \EE_n^-(s,t) & = & s\EE_{n-1}^-(s,t)+
    	t\EE_{n-1}^+(s,t)+
    	st D
    	\EE_{n-1}(s,t). 
    	\end{eqnarray}
    \end{lemma}
     
   \comment{ 
    \begin{proof} 
    \begin{align*}
     \EE_n^+(s,t) & = \frac{\EE_n(s,t)+(s-t)^n}{2} \\
                    & = \frac{(s+t)\EE_{n-1}(s,t) + 
                    2stD 
                    \EE_{n-1}(s,t)+(s-t)^n}{2} \\
                    & = (s-t)\frac{\EE_{n-1}(s,t)+
                    (s-t)^{n-1}}{2}
                    +t\EE_{n-1}(s,t)+stD
                    \EE_{n-1}(s,t)\\
                    & = (s-t) \EE_{n-1}^+(s,t)+ t\EE_{n-1}(s,t)+
                    stD 
                    \EE_{n-1}(s,t) \\
                    & =s\EE_{n-1}^+(s,t)+
                    t\EE_{n-1}^-(s,t)+
                    stD 
                    \EE_{n-1}(s,t)
      \end{align*}
    Hence, the first statement is proved. Similarly we get  
    the second one also. \end{proof}} 
\comment{
We use the above recurrence relations repeatedly to obtain 
	\begin{eqnarray} \EE_n^+(s,t)  
	& = & s\EE_{n-1}^+(s,t)+t\EE_{n-1}^-(s,t)+
	stD
	\EE_{n-1}(s,t)  \nonumber \\
	& = &(s^2+2st+t^2)\EE_{n-2}^+(s,t)	+ 
	4st  \EE_{n-2}^-(s,t)+  \label{mainrecBtype} \\
	&  &4st(s+t)D
	\EE_{n-2}(s,t)+2s^2t^2D ^2
	\EE_{n-2}(s,t)    \nonumber
	\end{eqnarray}
	In a similar manner we obtain
	\begin{eqnarray} \EE_n^-(s,t)& = 
	& 
	(s^2+2st+t^2)\EE_{n-2}^-(s,t)	+ 
	4st  \EE_{n-2}^+(s,t)+ \\
	& & 4st(s+t)D
	\EE_{n-2}(s,t)+2s^2t^2D ^2
	\EE_{n-2}(s,t)  \nonumber
	\end{eqnarray}}

\begin{theorem}
\label{thm:main_thm_for_Btype_excedance}
For all positive even integers $n = 2m$ with $n \geq 2$, 
$\EE_{n}^+(s,t)$ and $\EE_{n}^-(s,t)$ are gamma positive 
polynomials with the same center of symmetry $m$.
\end{theorem}

\comment{
     \begin{proof}
         For $n=2$, $\EE_{n}^-(s,t)=4st$
         	and 
         $\EE_{n}^+(s,t)=s^2+2st+t^2$ both of which are  
         gamma non-negative with centre of symmetry $1$. 
         Now, by induction we assume 
         for $n-2=2m-2$, both $\EE_{n-2}^-(s,t)$
         	and 
         $\EE_{n-2}^+(s,t)$ are Gamma-nonnegative with center 
         of symmetry $m-1$. Then, 
         using the earlier recurrence in \eqref{mainrecBtype} 
         we observe that each of the 
         four terms in \eqref{mainrecBtype} has center 
         of symmetry $m$.  Since 
         the terms all have the same center of symmetry, 
         the polynomial
         $\EE_{2m}^+(s,t)$ is gamma non-negative with 
         center of symmetry $m$, completing the proof. 
         In exactly similar manner, one can prove that 
         $\EE_{2m}^-(s,t)$ is gamma non-negative with
          center of symmetry $m$.	
     \end{proof}}

\begin{theorem}
\label{thm:sumof2foroddBtypeexcedance}
For all positive odd integers $n \geq 3$, $\EE_{n}^-(t)$ and 
$\EE_{n}^+(t)$ can be written as a sum of $2$ gamma positive polynomials.
\end{theorem}

\section{Type D Coxeter Group}

Let $\DD_n = \{ \sigma \in \BB_n: |\Negs(\sigma)|\mbox{ is even} \}$
denote the type-D Coxeter group.
For $\pi \in \DD_n$, define as before, $\Negs(\pi) =  \{\pi_i : i > 0, \pi_i < 0 \}$ to be 
the set of elements which occur with a negative sign.  
Define $\inv_D(\pi)=  \inv(\pi) 
+ | \{ 1 \leq i < j \leq n :- \pi_i > \pi_j \} |$. For Type-D Coxeter groups, 
we have the same definition of excedance as in type-B coxeter groups. 
Hence, define 
$\exc_D(\pi)= |\{ i \in [n]: \pi_{|\pi(i)|} > \pi_ i\}|+|\{ i \in [n]: \pi_i =- i\}|$,
$\antiexc_D(\pi)=n-\exc_D(\pi)$ and $\wexc_D(\pi)= |\{ i \in [n]: \pi_{|\pi(i)|} > \pi_ i\}|+
|\{ i \in [n]: \pi_i = i\}|$. Let $\DD^+_n \subseteq \DD_n$ denote the
subset of even length elements of $\DD_n$ and let $\DD_n^- = \DD_n - \DD_n^+$.
Define
\begin{eqnarray}
\label{eqn:Dn} 
   \FF_n(t) & = & \sum_{\pi \in \DD_n} t^{\exc_D(\pi)}  \mbox{ \hspace{2 mm} and \hspace{2 mm}} 
     	\FF_n(s,t)  =  \sum_{\pi \in \DD_n} t^{\exc_D(\pi)}
     	s^{\antiexc_D(\pi)} \\
     	 \label{eqn:DBn-Dn} 
     	 \FFBminusD_n(t) & = & \sum_{\pi \in \BB_n-\DD_n} 
     	 t^{\exc_D(\pi)} \mbox{ \hspace{1 mm} and \hspace{1 mm}} 
     	 \FFBminusD_n(s,t)  =  \sum_{\pi \in \BB_n-\DD_n}
     	 t^{\exc_D(\pi)}
     	 s^{\antiexc_D(\pi)}\  \\ 
     	\label{eqn:DBAn} 
     	\FF_n^+(t) & = & \sum_{\pi \in \DD_n^+} t^{\exc_D(\pi)}  
     	\mbox{ \hspace{3 mm} and \hspace{3 mm}} 
     	\FF_n^+(s,t)  =  \sum_{\pi \in \DD_n^+} t^{\exc_D(\pi)} 
     	s^{\antiexc_D(\pi)}\\  
     	\label{eqn:DBnAn} 
     	\FF_n^-(t) & = & \sum_{\pi \in \DD_n^-}
     	t^{\exc_D(\pi)} 
     	\mbox{ \hspace{3 mm} and \hspace{3 mm}}  
     	\FF_n^-(s,t)  =  \sum_{\pi \in \DD_n^-} t^{\exc_D(\pi)}
     	 s^{\antiexc_D(\pi)}\\
     	\FFsgnE_n(s,t) & = & \sum_{\pi \in \DD_n}
     	(-1)^{\inv_D(\pi)} t^{\exc_D(\pi)} s^{\antiexc_D(\pi)}
\end{eqnarray}
     	
Recall that Brenti showed that type-B excedances
and type-B ascents are equidistributed in $\BB_n$ (Theorem \ref{thm:brenti_fft}).

\begin{remark}
\label{rem:brenti_similar}
It is simple to give a similar bijection $T_n:\BB_n \mapsto \BB_n$ 
such that $\des_B(T_n(\pi))=\exc_B(\pi)$ and $\Negs((T_n(\pi))= \Negs(\pi)$. 
The bijection is very similar to the proof of  
\cite[Theorem 10.2.3]{lothaire-combin-words} and  \cite[Theorem 3.15]{brenti-q-eulerian-94} 
and so we omit it.
\end{remark}

\begin{remark}
\label{rem:brenti_restricts_type-D}
By Remark \ref{rem:brenti_similar}, as the number of negative entries is preserved by $T_n$, 
enumerating Type-B descents over $\DD_n$ is equivalent to enumerating  Type-B excedance over $\DD_n$ 
which is equivalent to enumerating Type-D excedance over $\DD_n$. That is, for all positive integers $n$,  
we have

$$\sum_{\pi \in \DD_{n}}t^{\des_{B}(\pi)}s^{\asc_{B}(\pi)}=
\sum_{\pi \in \DD_{n}}t^{\exc_{D}(\pi)}s^{\antiexc_{D}(\pi)}.$$
\end{remark}
 
Recall the operator $\displaystyle D = \left( \frac{d}{ds} + \frac{d}{dt} \right)$.  We 
prove the following recurrence relations.    
     
\begin{lemma}
\label{lem:rec_D_type_excedance}
For all positive integers $n$, the polynomials $\FF_n(s,t)$ satisfy the following recurrence 
relations.
\begin{eqnarray}
\label{eqn:rec_D_type}\
\FF_n(s,t) & = & s\FF_{n-1}(s,t)+
t\FFBminusD_{n-1}(s,t)+stD
\EE_{n-1}(s,t)\\\label{eqn:rec_B-D}
\FFBminusD_n(s,t) & = & t\FF_{n-1}(s,t)+
s\FFBminusD_{n-1}(s,t)+stD
\EE_{n-1}(s,t)
\end{eqnarray}
\end{lemma}
     
\begin{proof} 
We prove \eqref{eqn:rec_D_type} first.  Using Remark \ref{rem:brenti_restricts_type-D},
we evaluate
$\sum_{\pi \in \DD_{n}}t^{\des_{B}(\pi)}s^{\asc_{B}(\pi)}$ instead.
Clearly, $s\sum_{\pi \in \DD_{n-1}}t^{\des_{B}(\pi)}s^{\asc_{B}(\pi)}$ is the 
contribution  from $\pi \in \DD_{n}$ in which the letter `$n$' appears in the last position
and $t\sum_{\pi \in\BB_{n-1}-\DD_{n-1}}t^{\des_{B}(\pi)}s^{\asc_{B}(\pi)}$ is the contribution 
from $\pi \in \DD_{n}$ in which the letter $\ob{n}$ appears in the last position. 
The term $\displaystyle stD\sum_{\pi \in \DD_{n-1}}t^{\des_{B}(\pi)}s^{\asc_{B}(\pi)}$ 
accounts for all those permutations $\pi \in \DD_{n}$
in which the letter $n$ appears in all positions except the last position.  Likewise, the term
$\displaystyle stD\sum_{\pi \in \BB_{n-1}-\DD_{n-1}}t^{\des_{B}(\pi)}s^{\asc_{B}(\pi)}$ 
accounts  for all those permutations $\pi \in \DD_{n}$ in which the letter $\ob{n}$ appears 
at either the initial position or in the gaps in the permutations of $\BB_{n}-\DD_{n}$. 
Summing up all these we get 
\begin{eqnarray*}
\sum_{\pi \in \DD_{n}}t^{\des_{B}(\pi)}s^{\asc_{B}(\pi)}
 & = &
      s\sum_{\pi \in \DD_{n-1}}t^{\des_{B}(\pi)}s^{\asc_{B}(\pi)}+
      t\sum_{\pi \in \BB_{n-1}-\DD_{n-1}}t^{\des_{B}(\pi)}
      s^{\asc_{B}(\pi)}\\ 
  & &  + stD\sum_{\pi \in \DD_{n-1}}t^{\des_{B}(\pi)}s^{\asc_{B}(\pi)} + 
      stD\sum_{\pi \in \BB_{n-1}-\DD_{n-1}} t^{\des_{B}(\pi)}s^{\asc_{B}(\pi)}
\end{eqnarray*}
In an identical manner one can prove \eqref{eqn:rec_B-D}.  We omit the details.
\end{proof}

Using Lemma \ref{lem:rec_D_type_excedance}, we next show that  
the polynomials $\FF_{n}(s,t)$ and $\FFBminusD_{n}(s,t)$ are gamma positive.

\begin{theorem}
\label{thm:gamma-nonnegativite_for_D_type_excedance,n_even}
For all positive integers $n \geq 2$, $\EE_n^+(s,t)=\FF_{n}(s,t)$ and 
$\EE_n^-(s,t)=\FFBminusD_{n}(s,t)$. Hence, for all even natural numbers $n$,  
$\FF_{n}(s,t)$	and $\FFBminusD_{n}(s,t)$ are gamma positive polynomials with same center of 
symmetry $n/2$. For all odd positive integers $n \geq 3$, the univariate polynomials 
$\FF_{n}(t)$	and $\FFBminusD_{n}(t)$ can be written as a sum of $2$ gamma positive
polynomials.
\end{theorem}

\begin{proof}
It is simple to see that $\FF_2(s,t)=s^2+2st+t^2=\EE_2^+(s,t)$  and
$\FFBminusD_{2}(s,t)=4st=\EE_2^-(s,t)$. Further, recurrences 
\eqref{eqn:rec_D_type} and \eqref{eqn:rec_B-D}   for 
$\FF_n(s,t)$ and  $\FFBminusD_{n}(s,t)$ respectively are identical to 
the recurrences \eqref{rec.Btypeexcedanceeven} and \eqref{rec.Btypeexcedanceodd}. 
Thus, the proof of Theorem \ref{thm:main_thm_for_Btype_excedance} and 
Theorem \ref{thm:sumof2foroddBtypeexcedance} works here.
\end{proof}

We next give a recurrence relation satisfied by $\FF_n(s,t)$. 
   
\begin{lemma}
\label{lemma:Dexcnplus4fromDexcn}
For all positive integers $n \geq 2$,
\begin{eqnarray}
\label{eqn:Dexcnplus4fromDexcn}
\FF_{n+4}(s,t) & = & R_1(s,t) \FF_{n}(s,t) + 
 R_2(s,t) \FFBminusD_{n}(s,t)
+ R_3(s,t) D \EE_{n}(s,t) \nonumber \\ 
&  &  +R_4(s,t) D^2 \EE_{n}(s,t) + R_5(s,t) 
  D\EE_{n+1}(s,t) \nonumber  \\
& & + R_6(s,t)D^2\EE_{n+1}(s,t)+    R_7(s,t) 
D^2\EE_{n+2}(s,t)   \label{eqn:jumpfromnton+4} 
\end{eqnarray}  
where the following table lists $R_i(s,t)$ and its center of 
symmetry. 
   
$
\begin{array}{l|r}
R_i(s,t) & \cnos(R_i(s,t)) \\ \hline
R_1(s,t) = (s+t)^4+8st(s+t)^2+16(st)^2  &  2 \\ \hline
R_2(s,t) = 16st(s+t)^2  &  2 \\ \hline 
R_3(s,t) = 4st(s+t)^3+32(st)^2(s+t) &  5/ 2 \\ \hline
R_4(s,t) = 2(st)^2(s+t)^2 +8(st)^3&  3 \\ \hline
R_5(s,t) = 12(st)(s+t)^2 & 2 \\ \hline
R_6(s,t) = 8(st)^2(s+t) & 5/2 \\ \hline
R_7(s,t) = 2(st)^2 & 2  
\end{array}
$
   
Further, let 
\begin{eqnarray}\label{eqn:Sn+4}
S_{n+4}(s,t)&= &R_2(s,t) \FFBminusD_{n}(s,t)
+( R_3(s,t) D +R_4(s,t) D^2) \EE_{n}(s,t) \nonumber \\
 & & + (R_5(s,t) 
  D + R_6(s,t)D^2)\EE_{n+1}(s,t)+    R_7(s,t) 
  D^2\EE_{n+2}(s,t) 
\end{eqnarray}
Then, $S_{n+4}(s,t)$  is gamma positive with center of 
symmetry $\frac{1}{2}(n+4)$ and each gamma 
coefficient of $S_{n+4}(s,t)$ is even.
\end{lemma}
   
\begin{proof}
Applying Lemma \ref{lem:rec_D_type_excedance} twice, we get 
\begin{eqnarray}
\label{eqn:Dtypejumpby2}
 \FF_{n+4}(s,t) & = & (s+t)^2 \FF_{n+2}(s,t) + 4st \FFBminusD_{n+2}(s,t)
  + 4st(s+t) D \EE_{n+2}(s,t) \nonumber \\ 
  &  &  +2(st)^2 D^2 \EE_{n+2}(s,t) 
\end{eqnarray}
Further, 
\begin{eqnarray}
\FF_{n+2}(s,t) 
& = & (s+t)^2 \FF_{n}(s,t) + 4st \FFBminusD_{n}(s,t) + 4st(s+t) D \EE_{n}(s,t) \nonumber \\ 
&  &  +2(st)^2 D^2 \EE_{n}(s,t)   \label{eqn:1}  \\
\FFBminusD_{n+2(s,t)} 
&=&  4st \FF_{n}(s,t) + (s+t)^2 \FFBminusD_{n}(s,t) + 4st(s+t) D \EE_{n}(s,t) \nonumber \\ 
&  &  +2(st)^2 D^2 \EE_{n}(s,t)    \label{eqn:2} \\
D \EE_{n+2}(s,t)
&=& 2(s+t) \EE_{n}(s,t)+ 4st D\EE_{n}(s,t) +3(s+t)D\EE_{n+1}(s,t) \nonumber\\
& & + 2stD^2\EE_{n+1}(s,t)  \label{eqn:3}
\end{eqnarray}

We get \eqref{eqn:Dexcnplus4fromDexcn} by plugging in 
\eqref{eqn:1}, \eqref{eqn:2} and  \eqref{eqn:3}
in \eqref{eqn:Dtypejumpby2}.
Further, note that $S_{n+4}(s,t)$ is a sum of $6$ gamma 
positive polynomials, each with center of symmetry $\frac{1}{2}(n+4)$.   
Further, from the table, each $R_i(s,t)$ clearly has even gamma 
coefficients and hence, $S_{n+4}(s,t)$ also has even gamma 
coefficients.  
\end{proof}

Sivasubramanian in \cite{siva-weyl_linear} showed the following.

\begin{theorem}
\label{formula_uniivariate_signed_excedance}
 For all positive integers $n$,  
$\FFsgnE_n(t)=\begin{cases}
(1-t)^n & \text {if $n$ even }.\\
(1-t)^{n-1} & \text {if $n$ odd }.
\end{cases}$
\end{theorem} 
A slight modification of this result gives us the following
bivariate version of Theorem \ref{formula_uniivariate_signed_excedance}.
Since the proof is identical, we omit the details.

\begin{theorem}
\label{formula_bivariate_signed_excedance}
 For all positive integers $n$,  
$\FFsgnE_n(s,t)=\begin{cases}
(s-t)^n & \text {if $n$ even }.\\
s(s-t)^{n-1} & \text {if $n$ odd }.
\end{cases}$
\end{theorem}

We use Theorem \ref{formula_bivariate_signed_excedance} twice to 
get the following.
 
\begin{corollary} 
\label{cor:sgn_exc_typed_earlier}
For positive integers $n$, we have 
$\FFsgnE_{n+4}(s,t)=(s-t)^4\FFsgnE_{n}(s,t).$
\end{corollary}

We now come to the main Theorem of this section.
\begin{theorem}
\label{thm:d_type_exccedance_even_integers}.
For all even positive integers $n \geq 4$, $\FF_n^+(s,t)$ and $\FF_n^-(s,t)$ 
are gamma positive with center of symmetry $\frac{1}{2}n$. 
\end{theorem}
  
\begin{proof}
We induct on $n$ with the base cases being $n=4$ and $n=6$.
When $n=4$ and $n=6$ one can  check the following.
\begin{eqnarray*}
\FF_{4}^{+}(s,t) & = & s^4+16s^3t+62s^2t^2+16st^3+t^4  
  		=  (s+t)^4+12st(s+t)^2+32(st)^2, \\
\FF_{4}^{-}(s,t) & = &20s^3t+56s^2t^2+20st^3
  		= 20st(s+t)^2+16(st)^2, \label{eqn:dm+4_jump}\\
\FF_{6}^{+}(s,t) & = &  s^6+176s^5t+2647s^4t^2+
  		5872s^3t^3+2647s^2t^4+176st^5+t^6 \\ 
  		& = &(s+t)^6+170st(s+t)^4+1952(st)(s+t)^2+928(st)^3 , \\ 
\FF_{6}^{-}(s,t) & = & 182s^5t+2632s^4t^2+5892s^3t^3+
  		2632s^2t^4+182st^5 , \\
  		& = & 182st(s+t)^4+1904(st)^2(s+t)^2+992(st)^3. 
\end{eqnarray*}

By \eqref{eqn:jumpfromnton+4} and Corollary \ref{cor:sgn_exc_typed_earlier},
 \begin{eqnarray}
 \label{eqn:recdplus}
  \FF_{n+4}^+(s,t)&=& \frac{1}{2}\left(	\FF_{n+4}(s,t)+ \FFsgnE_{n+4}(s,t)\right) \nonumber \\
 	&=& \frac{1}{2}\left(R_1(s,t)\FF_{n}(s,t)+S_{n+4}(s,t) +(s-t)^4\FFsgnE_{n}(s,t)\right) \nonumber \\
  	&=&\frac{1}{2}\left([R_1(s,t)+(s-t)^4]\FF_{n}^+(s,t)\right)+ \frac{1}{2}\left([R_1(s,t)-(s-t)^4]
  	\FF_{n}^-(s,t)\right) \nonumber \\ 
  	& & +\frac{1}{2}S_{n+4}(s,t), \\
  	\label{eqn:recdminus}
  \FF_{n+4}^-(s,t)&=& \frac{1}{2}\left(	\FF_{n+4}(s,t)- \FFsgnE_{n+4}(s,t)\right) \nonumber \\
  &=&\frac{1}{2}\left([R_1(s,t)-(s-t)^4]\FF_{n}^+(s,t)\right)+ \frac{1}{2}\left([R_1(s,t)+(s-t)^4]
  \FF_{n}^-(s,t)\right) \nonumber \\ 
  & & +\frac{1}{2}S_{n+4}(s,t)
\end{eqnarray}

It is simple to see that $R_1(s,t)+(s-t)^4=2(s+t)^4+32(st)^2$ 
and $R_1(s,t)-(s-t)^4=16st(s+t)^2$ are both gamma positive with 
center of symmetry $2$ and both have even gamma coefficients.  
Let $n$  be even with $n >7$. By induction, both $\FF_{n}^{+}(s,t)$ and $\FF_{n}^{-}(s,t)$ 
are gamma positive with the same center of symmetry $\frac{1}{2}n$. 
Further, each of the three terms of \eqref{eqn:recdplus}
 have the same center of symmetry $\frac{1}{2}(n+4)$. Thus by 
Lemma \ref{lemma:Dexcnplus4fromDexcn},   
the polynomial $\FF_{n+4}^+(s,t)$ is gamma positive with center 
of symmetry $\frac{1}{2}(n+4)$. In an identical manner, one can prove that 
$\FF_{n+4}^-(s,t)$ is gamma positive with center of symmetry $\frac{1}{2}(n+4)$. 
\end{proof} 

\begin{theorem}
\label{thm:d_type_exccedance_sum_of_2,for_odd_integers}
For all odd positive integers $n \geq 5$, $\FF_n^+(t)$ and $\FF_n^-(t)$ 
can be written as sum of two gamma positive polynomials. 
\end{theorem}
 
\begin{proof} Let $n$ be an odd integer $\geq 5$. 
\begin{eqnarray}
  & & \FF_n^+(s,t) = \frac{1}{2} \left(\FF_{n}(s,t)+ \FFsgnE_n(s,t)  
\right)\nonumber
 	\\ & = & \frac{1}{2} \left(s\FF_{n-1}(s,t)+ \FFsgnE_n(s,t) + 
 	t\FFBminusD_{n-1}(s,t) \right) \nonumber + \frac{1}{2}stD 
 	\EE_{n-1}(s,t) \nonumber \\ \nonumber
 	& = & \frac{1}{2}\left(s\FF_{n-1}(s,t)+s\FFsgnE_{n-1}(s,t)+
 	t\FFBminusD_{n-1}(s,t)+stD 
 	\EE_{n-1}(s,t) \nonumber 
 	\right) \nonumber
\end{eqnarray}
 	Hence, $$\FF_n^+(t)=\frac{1}{2} \left({2\FF_{n-1}^+(t)+t\FFBminusD_{n-1}(t)+\{stD 
 		\EE_{n-1}(s,t)\}|_{s=1}}\right) .$$
 	Now the proof follows in similar lines to the proof of Theorem \ref{thm:sumof2forevenAtype}. 
\end{proof}

\section{Excedance in Even and odd Derangements }
Let $\SD_n \subseteq \SSS_n$ be the subset of derangements in $\SSS_n$.
Though descents and excedances in $\SSS_n$ are equidistributed, they are
not equidistributed in $\Der_n \cap \AAA_n$.  Recall $\ADE_n(t)$ as 
defined in \eqref{eqn:der_exc}.  
Shin and Zeng \cite{shin-zeng-eulerian-continued-fraction} and 
Sharesian and Wachs \cite{shareshian-wachs-eulerian-quasisymmetric} proved gamma
positivity of $\ADE_n(t)$.

\begin{theorem}[Shin and Zeng]
\label{thm:DerEn(t) Gamma-nonnegative}
For all positive integers $n$, $\ADE_n(t)$ is gamma positive.
\end{theorem}

Sun and Wang in \cite{sun-wang-grou-action-derang} gave an alternate 
proof of Theorem  \ref{thm:DerEn(t) Gamma-nonnegative} based on 
a semi bijective variant of {\it valley hopping} that we call {\it cyclic
valley hopping}.  It can be checked that the cyclic valley 
hopping proof preserves cycle type and hence sign. 
Thus, we get the following refinement of Theorem 
\ref{thm:DerEn(t) Gamma-nonnegative}.  We give a different proof 
as our proof gives more refined information. 

\begin{theorem}
\label{thm:even_derangs_exc}
The polynomials $\ADE_n^+(t)$ and $\ADE_n^-(t)$ are gamma positive
for all positive integers $n$.
\end{theorem}

\subsection{Gamma-nonnegativity for even and odd derangements}
   
For $\pi \in \SSS_n$, let $\cyctyp(\pi)=1^{m_1}2^{m_2}\ldots k^{m_k}$
where we have $m_i$'s are positive integers and $\pi$ has $m_i$ cycles of length $i$.
We denote by $\lambda(\pi) \vdash n$, the partition (of the integer $n$) with $m_i$ 
parts having size $i$ and also denote this as $\cyctyp(\pi) = \lambda$.
It is well known that conjugacy classes in $\SSS_n$ are indexed by 
partitions $\lambda \vdash n$ and the conjugacy class 
$C_{\lambda} = \{ \pi \in \SSS_n: \cyctyp(\pi) = \lambda\}$.
For $\lambda \vdash n$, define 
\begin{equation}
\label{eqn:defn_exc_conjugacy}
\ADE_{\lambda}(t)=\sum_{\pi \in C_{\lambda}}t^{\exc(\pi)}.
\end{equation}

We start by summing excedance over cyclic permutations (which are all
permutations $\pi$ with $\cyctyp(\pi) = n$).
\begin{lemma}
\label{lemma:excedanceforonecycle}
Let $n \geq 2$ and consider the partition $\mu = n$.  Then the following holds.
\begin{eqnarray}
\label{eqn:excedanceforonecycle}
\ADE_{\mu}(t)=\sum_{\pi \in \SSS_n,\cyctyp(\pi)=n }t^{\exc(\pi)}=tA_{n-1}(t)
\end{eqnarray}
\vspace{-2 mm}
Thus $\ADE_{\mu}(t)$ is gamma positive with center of symmetry $\frac{1}{2}n$.
\end{lemma}

\vspace{2 mm} 

\begin{proof}
Let $\pi = a_1,a_2, \ldots, a_{n-1} \in \SSS_{n-1}$.
Define  $f:\SSS_{n-1} \mapsto C_{\mu}$ by 
$f(\pi)= (1, n+1-a_{1},
n+1-a_{2},\ldots,n+1-a_{n-1})$, 
where $a_1, a_2, 
\ldots, a_{n-1}$ is in one-line notation and 
$(1, n+1-a_{1}, n+1-a_{2}, \ldots, n+1-a_{n-1})$ 
is in cyclic notation. Clearly, $\pi$ has $k$ descents if and 
only if $n-a_1, n-a_2, \ldots, n-a_{n-1}$ has $k$ ascents.
This happens if and only if $f(\pi)=(1, n+1-a_{1}, n+1-a_{2}, \ldots,
n+1-a_{n-1})$ has $k+1$ excedances. 
$f$ is a bijection with $f^{-1} = g$ where $g: C_{\mu} \mapsto \SSS_{n-1}$ 
is defined by $g((1,a_1,a_2,\ldots,a_{n-1}))=n+1-a_1, 
n+1-a_2, \ldots, n+1-a_{n-1}$.  $A_{n-1}(t)$ is gamma positive with 
center of symmetry $\frac{1}{2}(n-2)$ and hence $\ADE_{\mu}(t)=
tA_{n-1}(t)$ has center of symmetry $\frac{1}{2}n$, completing
the proof.
\end{proof}

Let $\lambda \vdash n$.  We use two 
notations for $\lambda$.  
In exponential 
notation, let $\lambda = 1^{m_1},2^{m_2}, \ldots, k^{m_k}$.
This means $\lambda$ has $m_i$ parts equal to $i$ for $1 \leq i \leq k$.
We also use $\lambda = \lambda_1, \lambda_2, \ldots, \lambda_{\ell}$ to
denote that the parts of $\lambda$ are $\lambda_i$ for $1 \leq i \leq \ell$.
Define $\binom{n}{\lambda}$  to be the multinomial coefficient
$\binom{n}{\lambda_1, \ldots, \lambda_{\ell}}$.  
We next consider $C_{\lambda}$, the conjugacy class of $\SSS_n$ indexed by $\lambda$.
It is well known that $C_{\lambda}$ contains all $\pi \in \SSS_n$ with 
$\cyctyp(\pi) = \lambda$.   We start with the case when 
$\lambda = n_1,n_2$ with both $n_1,n_2 \geq 1$.
We do not need $n_1$ and $n_2$ to be distinct.  

\begin{lemma}
\label{lemma:excedancefortwocycle}
Let $\lambda \vdash n$ with $\lambda = n_1,n_2$ where $n_1, n_2 \geq 1$.
Then,
\begin{eqnarray}
\label{eqn:excedanceforrtwocycle}
\sum_{\pi \in C_{\lambda} } t^{\exc(\pi)} & =& 
\displaystyle \frac{ \binom{n}{\lambda} }{\prod_{i=1}^k (m_i!)} 
\left(\sum_{\pi \in \SSS_{n_1},\cyctyp(\pi)=n_1}t^{\exc(\pi)} \right) \times 
\left( \sum_{\pi \in \SSS_{n_2}, \cyctyp(\pi)=n_2}t^{\exc(\pi)} \right) \\
\label{eqn:excedance_two_cycles_eulerian}
& = & 
\displaystyle \frac{ \binom{n}{\lambda} }{\prod_{i=1}^k (m_i!)} 
\left[ t A_{n_1-1}(t) \right] \times \left[ t A_{n_2-1}(t) \right]
\end{eqnarray}
Thus, $\sum_{\pi \in C_{\lambda}} t^{\exc(\pi)}$ is gamma 
positive with center of symmetry $\frac{1}{2}n$ if both $n_1, n_2 > 1$. 
\end{lemma}
   
\begin{proof}
Let $T = \{{\pi \in \SSS_{n_1}:\cyctyp(\pi)=n_1 }\} 
\times\{{\pi \in \SSS_{n_2}:\cyctyp(\pi)=n_2 }\}$.
Let $\pi \in C_{\lambda}$ with $\pi = (a_1,a_2,\ldots,a_{n_1}), 
(b_1,b_2, \ldots,b_{n_2})$ (written in cycle notation).
Define $f: C_{\lambda} \mapsto T$ as 
$f(\pi)=
(a_1',a_2',\ldots,a_{n_1}'),(b_1',b_2,\ldots,b_{n_2}')$
where $(a_1',a_2',\ldots,a_{n_1}')$ is the order preserving
restriction of the set $A = \{a_1,a_2, \ldots, a_{n_1}\}$ to 
the set $[n_1]$.  That is we map the least element of $A$ to
$1$, the second smallest element to 2 and so on.  Similarly
$(b_1',b_2,\ldots,b_{n_2}')$ is the order preserving
restriction of $B = \{b_1, b_2, \ldots, b_{n_2} \}$ to $[n_2]$.
This is an 
$\displaystyle \frac{ \binom{n}{\lambda} }{\prod_{i=1}^k (m_i!)} $
to $1$ mapping. As excedances
are preserved under this map, the proof of 
\eqref{eqn:excedanceforrtwocycle} is complete.  
\eqref{eqn:excedance_two_cycles_eulerian} follows from 
Lemma \ref{lemma:excedanceforonecycle}.  Further, if both
$n_1, n_2 > 1$, then both $A_{n_1 -1}(t)$ and $A_{n_2-1}(t)$
have positive degree and thus the center of symmetry will 
clearly be $n/2$.
\end{proof}

\begin{remark}
\label{rem:integrality}
For a positive integer $n$, let $\lambda \vdash n$ with $\lambda = 1^{m_1}, 2^{m_2}, \ldots, k^{m_k}$.  
Then,
$\displaystyle \frac{ \binom{n}{\lambda} }{\prod_{i=1}^k (m_i!)}$ is a positive 
integer.  Indeed, it counts the number of ways to partition the set $[n]$ into parts with
sizes $\lambda_i$.
\end{remark}

\begin{theorem}
\label{thm:excedancewithcycletype}
Let $\lambda \vdash n$ with $\lambda =  1^{m_1}2^{m_2}\ldots k^{m_k}$. 
If $m_1=0$, then
\begin{equation}
\label{eqn:excedanceforkcycleswith1notpart}
\ADE_{\lambda}(t)= \sum_{\pi \in C_{\lambda}} t^{\exc(\pi)}= 
\displaystyle \frac{ \binom{n}{\lambda} }{\prod_{i=1}^k (m_i!)} 
\prod_{j=2}^{k} \left[ tA_{j-1}(t) \right]^{m_j}.
\end{equation}
Hence, $\ADE_{\lambda}(t)$ is gamma positive with center of symmetry 
 $\frac{1}{2}n$.
If $m_1>0$, then 
\begin{equation}
\label{eqn:excedanceforkcycleswith1apart}
\ADE_{\lambda}(t)= \sum_{\pi \in C_{\lambda}}t^{\exc(\pi)}=
\displaystyle \frac{ \binom{n}{\lambda} }{\prod_{i=1}^k (m_i!)} 
\prod_{j=2}^{k} [tA_{j-1}(t)]^{m_j}.
\end{equation}
Hence, $\ADE_{\lambda}(t)$ is a gamma positive
polynomial with center of symmetry $\frac{1}{2}(n-m_1)$. 
\end{theorem}
   
\begin{proof}
Is an easy application of Lemma \ref{lemma:excedancefortwocycle}.
By Lemmas \ref{lem:prod_bivariate_gamma_nonneg} and \ref{lemma:excedanceforonecycle}, 
$\prod_{j=2}^{k} (tA_{j-1}(t))^{m_j}$ 
is gamma positive with center of symmetry $\sum_{j=2}^{k}\frac{1}{2}(jm_j)$.
By Remark \ref{rem:integrality}, the multiplication factor is 
a positive integer and so gamma positivity is preserved.
We only need to note that when $m_1 = 0$, 
$\sum_{j=2}^{k}\frac{1}{2}(jm_j)=\frac{1}{2}n$.
Likewise, when $m_1 >0$, 
$\sum_{j=2}^{k}\frac{1}{2}jm_j=\frac{1}{2}(n-m_1)$.  
\end{proof}

\vspace{2 mm}

The following is our refinement of Theorem 
\ref{thm:DerEn(t) Gamma-nonnegative}.

\begin{theorem}
\label{thm:evenderangsexc}
For all positive integers $n$, the polynomials 
$\ADE_n(t), \ADE_n^+(t)$ and $\ADE_n^-(t)$ 
are gamma positive with centers of 
symmetry $\frac{n}{2}$.   
\end{theorem}
\begin{proof}
We denote those partitions $\lambda \vdash n$, with no part being 1 as 
$1 \not\in \lambda$.  
Clearly $\SD_n = \biguplus_{\lambda \vdash n, 1 \not\in \lambda} C_{\lambda}$.
Thus, by Theorem \ref{thm:excedancewithcycletype}, summing over all 
conjugacy classes indexed by $\lambda \vdash n$ with 
$1 \not\in \lambda$, we get gamma positivity of $\ADE_n(t)$.

If $\lambda \vdash n$ with $\lambda = 1^{m_1}2^{m_2} \ldots k^{m_k}$, 
then clearly, if $\pi \in C_{\lambda}$, then $\pi \in \AAA_n$ iff 
$n-\sum_{j=1}^{k} m_j$ is even.  For $\lambda \vdash n$ define
$\sgn(\lambda) = \sgn(\pi)$ for any $\pi \in C_{\lambda}$.
Thus, summing over conjugacy classes $C_{\lambda}$ with 
$1 \not\in \lambda$ and with $\sgn(\lambda) = 1$ gives us
gamma positivity of $\ADE_n^+(t)$.  Likewise, summing over
conjugacy classes $C_{\lambda}$ with $1 \not\in \lambda$ and 
with $\sgn(\lambda) = -1$ gives us gamma positivity of $\ADE_n^-(t)$.
\end{proof}
   
Define $\SD_{n,i} \subseteq \SSS_n$ denote the permutations with $i$ fixed points. 
Let  $\SD_{n,i} ^+= \SD_{n,i} \cap \AAA_n$ and $\SD_{n,i} ^-= \SD_{n,i} \cap (\SSS_n-\AAA_n)$.
Define,
\begin{eqnarray}
\label{eqn:der_exc_rep}
\ADE_{n,i}(t) & = & \sum_{\pi \in \SD_{n,i}} t^{\exc(\pi)} 
\mbox{ 
	\hspace{ 3 mm}
	and 
	\hspace{ 3 mm}
}
\ADE_{n,i}(s,t)  =  \sum_{\pi \in \SD_{n,i}} t^{\exc(\pi)} s^{\antiexc(\pi)-1}, \\
\label{eqn:der_exc_even_rep}
\ADE_{n,i}^{+}(t) & = & \sum_{\pi \in \SD_{n,i}^+} t^{\exc(\pi)} 
\mbox{ 
	\hspace{ 3 mm}
	and 
	\hspace{ 3 mm}
}
\ADE_{n,i}^{+}(s,t)  =  \sum_{\pi \in \SD_{,i}^+} t^{\exc(\pi)} s^{\antiexc(\pi)-1}, \\
\label{eqn:der_exc_odd_rep}
\ADE_{n,i}^{-}(t) & = & \sum_{\pi \in \SD_{n,i}^-} t^{\exc(\pi)}
\mbox{ 
	\hspace{ 3 mm}
	and 
	\hspace{ 3 mm}}
\ADE_{n,i}^{-}(s,t)  =  \sum_{\pi \in \SD_{n,i}^-} t^{\exc(\pi)}s^{\antiexc(\pi)-1}.
\end{eqnarray}

\begin{theorem}
\label{thm:gammapositivederangementwithfixedpoints}
For all positive integers $n,i$ with $0 \leq i \leq n$, the polynomials 
$ \ADE_{n,i}^+(t)$ and $\ADE_{n,i}^-(t)$ are gamma positive with
center of symmetry $\frac{1}{2}(n-i)$.
\end{theorem}
 
\begin{proof} 
Sum Theorem \ref{thm:excedancewithcycletype} over the conjugacy 
classes $C_{\lambda}$ with $m_1 = i$ in $\lambda$ and with
$\sgn(\lambda) = \pm 1$ (with $\sgn(\lambda)$ as defined in the proof
of Theorem \ref{thm:evenderangsexc}).
\end{proof}

Recall that  $\cyc(w)$ denotes the number of cycles of a
permutation $w$.  We are now ready to prove the Main
result of this Section. 

\begin{proof} (Of Theorem \ref{thm:main_exc})
We prove separately for the two statistics $\inv(\pi)$ and $\cyc(\pi)$. \\
$\underline{\stat(\pi)=\inv(\pi)}:$ Shin and Zeng \cite{shin-zeng-eulerian-continued-fraction}.  
showed that  
\begin{eqnarray}
\label{eqn:inv_der_shin_zeng}
\sum_{\pi \in \SD_n} q^{\inv(\pi)}t^{\exc(\pi)}=\sum_{i=0}^{\floor{n/2}}
b_{n,i}(q)t^{i}(1+t)^{n-2i}.
\end{eqnarray} 
where $b_{n,i}(q)=\sum_{\pi \in \SD_n(i)}q^{\inv(\pi)}$. Here,
$\SD_n(i)$ consists of all elements of $\SD_n$ with exactly $i$ excedances and no double excedances. 
Consider both terms on either side of the equaliy as polynomials in $\mathbb{R}[t][q]$. 
The coefficient of $q^{2r}$ and $q^{2r+1}$ for each $r \geq 0$ on 
either side are the same. 
Hence, \eqref{eqn:inv_der_shin_zeng}
factors nicely for $\SD_n^+$ and  $\SD_n^-$. 

\noindent
$\underline{\stat(\pi)= \cyc(\pi)}:$ Shin and Zeng in \cite{shin-zeng-eulerian-continued-fraction}  also showed
\begin{eqnarray} \label{eqn:cyclenumberexcedance}
\sum_{\pi \in \SD_n} q^{\cyc(\pi)}t^{\exc(\pi)}=\sum_{i=0}^{\floor{n/2}}
f_{n,i}(q)t^{i}(1+t)^{n-2i}  
\end{eqnarray} 
where $f_{n,i}(q)=\sum_{\pi \in \SD_n(i)}q^{c(\pi)}$. 
Again comparing the coefficients of $q^{2r}$ and 
$q^{2r+1}$ in both sides of \eqref{eqn:cyclenumberexcedance} 
completes the proof.  
\end{proof}

\bibliographystyle{acm}
\bibliography{main}
\end{document}